\documentclass[11pt]{amsart}

\usepackage[margin=3.5cm]{geometry}
\usepackage{amssymb, amsmath}
\usepackage{amscd}
\usepackage[alwaysadjust]{paralist}
\usepackage{hyperref}
\usepackage[capitalize]{cleveref}
\usepackage{enumitem}
\usepackage{esint}
\usepackage{graphicx} 
\usepackage[rgb]{xcolor} 
\usepackage{verbatim}

\usepackage{amsthm}

\newtheorem{cor}[equation]{Corollary}
\newtheorem{lemma}[equation]{Lemma}
\newtheorem{prop}[equation]{Proposition}
\newtheorem{theorem}[equation]{Theorem}

\theoremstyle{definition}

\def\IN{\mathbb N}
\def\IR{\mathbb R}

\def\eps{\varepsilon}

\newcommand{\area}{\operatorname{area}}

\begin{document}
\title{A remark on the rigidity of the first conformal {S}teklov eigenvalue}

\author{Henrik Matthiesen}
\address{Henrik Matthiesen: Department of Mathematics, University of Chicago,
5734 S. University Ave, Chicago, Illinois 60637}
\email{hmatthiesen@math.uchicago.edu}

\author{Romain Petrides}
\address{Romain Petrides, Universit\'e de Paris, Institut de Math\'ematiques de Jussieu - Paris Rive Gauche, b\^atiment Sophie Germain, 75205 PARIS Cedex 13, France}
\email{romain.petrides@imj-prg.fr}
\date{\today}

\begin{abstract} 
We show rigidity of the first conformal Steklov eigenvalue on annuli and M{\"o}bius bands.
The proof relies, among others, on uniqueness results due to Fraser--Schoen, a compactness theorem of the second named author, and recent work of the authors on asymptotic control of Steklov eigenvalues in glueing constructions.
\end{abstract}

\maketitle

For a compact and connected surface $\Sigma$ with smooth, non-empty boundary that is endowed with a smooth metric $g$, 
we denote by $\sigma_1(\Sigma,g)$ the first positive Steklov eigenvalue.
This is the smallest non-zero eigenvalue of the Dirichlet-to-Neumann operator given by $Tu = \partial_\nu \hat u$ for $u \in C^\infty(\partial \Sigma)$, where $\hat u \in C^\infty(\Sigma)$ denotes the harmonic extension of $u$ to all of $\Sigma$, and $\nu$ is the outward-pointing normal vector field along $\partial \Sigma$.
We also write $L_g(\partial \Sigma)$ for the length of the boundary of $\Sigma$.

We show the following rigidity result for the first Steklov eigenvalue in a conformal class on annuli and M{\"o}bius bands.

\begin{theorem} \label{thm_annuli}
Let $\Sigma$ be an annulus or a M{\"o}bius band endowed with a flat metric $g$ then
there is $\omega \colon \Sigma \to \IR$ such that
\begin{equation} \label{eq_gap}
\sigma_1 \left(\Sigma,e^{2\omega} g \right) L_{e^{2 \omega} g}(\partial \Sigma) > 2 \pi.
\end{equation}
\end{theorem}

Note that 
$$
\sigma_1(\mathbb{D}^2, \xi) L_\xi(\partial \mathbb{D}^2) = 2\pi,
$$
where $\xi$ is the standard metric.
Moreover, the $\xi$ maximizes the first Steklov eigenvalue among all metrics on the flat disk thanks to Weinstock's inequality, \cite{weinstock}.

It is conjectured that \cref{thm_annuli} holds for any $\Sigma$ not diffeomorphic to the disk.
We remark that the analogous result for the first Laplace eigenvalue on closed surfaces is known by work of the second author \cite{petrides}, see also \cite{petrides-2} for the higher dimensional case.

As a consequence of the work in \cite{petrides-3} and \cref{thm_annuli} we also obtain

\begin{theorem} \label{thm_harmonic}
Let $\Sigma$ be an annulus or a M{\"o}bius band endowed with a flat metric $g$ then there 
there is a smooth function $\omega \colon \Sigma \to \IR$ such that
$$
\sigma_1 \left(\Sigma,e^{2\omega} g \right) L_{e^{2 \omega} g}(\partial \Sigma) 
\geq
\sigma_1 \left(\Sigma,e^{2\tau} g \right) L_{e^{2 \tau} g}(\partial \Sigma) 
$$
for any smooth $\tau \colon \Sigma \to \IR$.
In particular, there is a free boundary harmonic map $\Phi \colon (\Sigma,g) \to \mathbb{B}^N$ by first eigenfunctions.
Here, $N \leq 3$ if $\Sigma$ is an annulus and $N \leq 4$ if $\Sigma$ is a M{\"o}bius band.
\end{theorem} 

The bounds on $N$ follow from \cite[Corollary 1.4.1]{KKP}, see also \cite[Theorem 2.3 and Theorem 2.4]{fs}.

We remark that it follows from our work \cite{MP} that for any compact $\Sigma$ with smooth non-empty boundary, that is not diffeomorphic to a disk, there is an entire family of specifically degenerating conformal classes for which the analogue of \cref{thm_annuli} holds, see \cref{thm_mp} below for the corresponding version for annuli.
In the case of annuli and M{\"o}bius bands the strong rigidity results by Fraser--Schoen allow us to propagate this to all conformal classes.

In a similar direction Karpukhin--Stern, also relying on \cite{GL}, have recently obtained versions of \cref{thm_annuli} and \cref{thm_harmonic} for a very different set of conformal classes, \cite{KS}.
They show the analogous result holds for some $\Sigma \subset (S,g)$, where $S$ is a closed surface and $\Sigma$ has many boundary components.

Our proof of \cref{thm_annuli} combines several ingredients.
The existence of maximizing metrics for the normalized first Steklov eigenvalue in a conformal class under the gap assumption \eqref{eq_gap} proved by the second author \cite{petrides-3}, the connection of extremal metrics for the Steklov problem and free boundary minimal surfaces,
a rigidity result by Fraser--Schoen for free-boundary minimal annuli and M{\"o}bius bands \cite{fs}, the asymptotic computation of the Steklov eigenvalues of $\overline{B}_1 \setminus B_\eps$ \cite{dittmar}, and the glueing construction from the recent work of the authors \cite{MP}.

We begin by recalling the connection between extremal metrics for Steklov eigenvalues and free-boundary minimal surfaces.
For simplicity we only state a simplified version sufficient for our purposes.

\begin{theorem}[{\cite[Proposition 5.2]{fs}}] \label{thm_connection}
Let $g$ be a smooth metric on $\Sigma$ such that
$$
\sigma_1(\Sigma,g) L_g(\partial \Sigma) \geq \sigma_1(\Sigma,h) L_h(\partial \Sigma)
$$
for any metric $h \in U$ where $U$ is an open neighborhood of $g$ in the $C^\infty$-topology.\footnote{In \cite{fs} it is assumed that $g$ is globally maximizing but the proof does not use this.}
Then there is a branched, conformal, minimal immersion $\Phi \colon (\Sigma,g) \to \mathbb{B}^N$ with free boundary by first eigenfunction for some $N \geq 2$.
\end{theorem}

Next, we state the rigidity results by Fraser--Schoen.

\begin{theorem}[{\cite[Theorem 1.2 and Theorem 1.4]{fs}}] \label{thm_fs}
Let $\Phi \colon \Sigma \to B^N$ be a minimal, free boundary immersion by first Steklov eigenfunctions.
If $\Sigma$ is an annulus, then $\Sigma$ is homothetic to the critical catenoid.
If $\Sigma$ is a M{\"o}bius band, then $\Sigma$ is homothetic to the criticial M{\"o}bius band.
\end{theorem}

We have the following existence and compactness result by the second named author.

\begin{theorem}[{\cite[Theorem 2]{petrides-3}}] \label{thm_existence}
Let $(\Sigma$,g) be a compact surface with non-empty boundary such that
$$
\sup_{\omega} (\sigma_1(\Sigma,e^{2 \omega} g)) L_{e^{2 \omega}g}(\partial \Sigma) > 2 \pi,
$$
then there is a smooth function $\tau \colon \Sigma \to \IR$ such that
\begin{equation} \label{eq_max}
\sigma_1(\Sigma,e^{2 \tau} g)) L_{e^{2 \tau}g}(\partial \Sigma)
=
\sup_{\omega} (\sigma_1(\Sigma,e^{2 \omega} g)) L_{e^{2 \omega}g}(\partial \Sigma)).
\end{equation}
Moreover, for any sequence of $(g_k)_{k \in \IN}$ of metrics with
\begin{equation} \label{eq_no_bubbling}
\liminf_{k \to \infty}  \sup_{\omega} (\sigma_1(\Sigma,e^{2 \omega} g_k)) L_{e^{2 \omega}g_k}(\partial \Sigma) > 2 \pi,
\end{equation}
the sequence $(\tau_k)_{k \in \IN}$ as in \eqref{eq_max} is smoothly precompact.
\end{theorem}

We remark that the last item is not explicitly stated in \cite[Theorem 2]{petrides-3}.
Instead it easily follows from the characterization of these maximizing metrics in terms of free boundary harmonic maps.
Along a sequence enjoying \eqref{eq_no_bubbling} there can not be any bubbling of these harmonic maps.
This is handled in \cite{petrides-3} even under much weaker assumptions.

We also have the following comparison result for Steklov eigenvalues.

\begin{theorem}[{\cite{dittmar}, see also \cite[Example 4.2.5.]{gp}}] \label{thm_comp_flat}
For $\eps>0$ sufficiently small, we have 
$$
\sigma_1(\overline{B_1} \setminus B_\eps) L(\partial (\overline{B_1} \setminus B_\eps)) > 2 \pi.
$$
\end{theorem}

We need the analogous result for M{\"o}bius bands.
Let $M_\eps$ the M{\"o}bius band obtained as follows.
We glue together two copies $A_1$ and $A_2$ of $\overline{B}_1 \setminus B_\eps$ along $\partial B_\eps$ and identify points by the involution given by
$\iota (x_1)=-x_2$, where $x_1 \in A_1$ and $x_2$ denotes the point with the same coordinates as $x_1$ but in $A_2$.
Note that the metric on $M_\eps$ is only Lipschitz, but it can easily be approximated by a sequence of smooth metrics such that the length of the boundary and the first Steklov eigenvalue converge.

\begin{prop}
We have that
$$
\sigma_1(M_\eps) L(\partial M_\eps) > 2\pi
$$
for $\eps>0$ sufficiently small.
\end{prop}

The argument is analogous to (and in fact easier than) the proof of \cref{thm_comp_flat}.
By the symmetries of $M_\eps$ it suffices to find the eigenvalues of two mixed problems on $\overline{B}_1 \setminus B_\eps$.
These are the problems with Steklov 
conditions along $\partial B_1$ and Dirichlet or Neumann conditions along $\partial B_\eps$.
One then only has to check which linear cominations of eigenfunctions corresponding to the same eigenvalue are invariant under the involution $\iota$.

We also need that a similar results holds near the other end of the moduli space. 
This result is much more subtle and was only very recently obtained in much greater generality by the authors in \cite[Theorem 1.3]{MP}.
While this is not explicitly stated there it easily follows from the specific construction, which we briefly recall now.

Let $\mathbb{D}^2$ be the flat disk and 
$$
\Omega_\eps = \left\{ (x,y) \in \IR^2 : \eps \exp\left(- \frac{1}{\eps^\alpha} \right) \leq y \leq \eps, \ -\frac{y^2}{2} \leq x \leq \frac{y^2}{2} \right\}
$$
for some $\alpha \in (\frac{2}{5},\frac{1}{2})$.
We then obtain $\Sigma_\eps$ by glueing $\Omega_\eps$ along the two sides $\{y=\eps\}$ and $\left\{y = \eps \exp\left(- \frac{1}{\eps^\alpha} \right)\right\}$ to $\partial \mathbb{D}^2$ along disjoint intervals of the corresponding length in the boundary in conformal coordinates.
We denote by $g_\eps$ the canonically induced metric from this construction. 
Note that $\Sigma_\eps$ can be an annulus or a M{\"o}bius band depending on the orientations chosen in the glueing procedure.

\begin{theorem}[{\cite[Theorem 1.3]{MP}}] \label{thm_mp}
There are $\omega_\eps \colon \Sigma_\eps \to \IR$ such that
$$
\sigma_1(\Sigma_\eps, e^{2 \omega_\eps} g_\eps) L_{e^{2 \omega_\eps}g_\eps}(\partial \Sigma_\eps) > 2 \pi
$$
for $\eps$ sufficiently small.
\end{theorem}

We use the following simple lemma to control the conformal class of the surfaces $\Sigma_\eps$.

\begin{lemma} \label{lem_conformal}
Let $\phi \colon [0,r] \times [0,1] \to [0,R] \times [0,1]$ be conformal embedding, smooth in the interior, such that $\phi([0,r] \times \{i\}) \subset [0,R] \times \{i\}$, for $i=0,1$.
Then we have that $r \leq R$.
\end{lemma}

\begin{proof}
Since $\phi$ is conformal, we have that 
$$
\phi^\star (dx^2+dy^2) = e^{2 \omega} (dx^2 + dy^2)
$$
for a function $\omega$ that is smooth in the interior.
By assumption, we have that
$$
1 
\leq \left( L (\phi(\{x\} \times [0,1])) \right)^2 
= \left(\int_0^1 e^{\omega(x,y)} dy \right)^2
\leq \int_0^1 e^{2 \omega(x,y)} dy,
$$
where we have used Jensen's inequality and note that this remains valid also if $\phi(\{x\} \times [0,1])$ is not a rectifiable curve.
This implies that
$$
r 
= \int_0^r  dx
\leq \int_0^r \int_0^1 e^{2 \omega(x,y)} dy dx
=\area(\phi) \leq R
\qedhere
$$
\end{proof}

We can now prove the following

\begin{lemma} \label{lem_conformal_2}
Assume that $\Sigma_\eps$ is an annulus and
let $\Phi \colon \overline{B}_1 \setminus B_r \to \Sigma_\eps$ be a conformal homeomorphsim, which is smooth in the interior.
Then we have that $r \to 1$ as $\eps \to 0$.
\end{lemma}

\begin{proof}
Let 
\begin{equation} \label{eq_incl}
\phi_0 \colon [0,r_0] \times [0,1] \to \Omega_\eps \subseteq \left[-\frac{\eps^2}{2},\frac{\eps^2}{2}\right] \times \left[\eps \exp\left( - \frac{1}{\eps^\alpha} \right), \eps\right] 
\vspace{0.2cm}
\end{equation}
be a conformal homeomorphism which is smooth in the interior and on the boundary maps vertices to vertices\footnote{Note that even though the M{\"o}bius group only acts simply three-transitively on the boundary, the additional parameter $r_0$ allows us to do this.}
such that 
we have that $\phi_0([0,r_0] \times \{0\}) \subset  \left[-\frac{\eps^2}{2}, +\frac{\eps^2}{2} \right] \times \left\{ \eps \exp\left( - \frac{1}{\eps^\alpha} \right) \right\} $
and
$\phi_0([0,r_0] \times \{1\}) \subset \left[-\frac{\eps^2}{2}, +\frac{\eps^2}{2} \right] \times \{ \eps\}$.
Also note that the inclusion of $\Omega_\eps$ into the rectangle in \eqref{eq_incl} is isometric, in particular conformal.
After scaling the codomain by $\eps- \eps \exp\left( - \frac{1}{\eps^\alpha} \right) \sim \eps$ we  
can apply \cref{lem_conformal} to $\phi_0$ and find that
\begin{equation} \label{eq_conf_1}
r_0 \leq C \eps
\end{equation}
for some fixed constant $C>0$.

We now conformally parametrize the universal covering $\widetilde \Sigma_\eps$ of $\Sigma_\eps$ by
$(-\infty,\infty) \times [0,1]$ with deck transformations generated by $(x,y) \mapsto (x+R,y)$ for $R>0$, which uniquely determines $R$.
Using $\phi_0$ we then find a conformal embedding
$$
\phi \colon [0,r_0^{-1}] \times [0,1] \to \Sigma_\eps 
$$
given explicitly by
$$
\phi(x,y)=\phi_0(r_0 y,r_0 x) \in \Omega_\eps \subseteq \Sigma_\eps.
$$
We then lift $\phi$ to a map
$$
\tilde \phi \colon [0,r_0^{-1}] \times [0,1] \to  [0,R] \times [0,1] \subset \widetilde \Sigma_\eps
$$
with
$
\tilde \phi([0,r_0^{-1}] \times \{i\}) \subset  [0,R] \times \{i\}.
$
\cref{lem_conformal} applied to $\tilde \phi$ gives that 
$$
r_0^{-1} \leq R,
$$
from which the claim easily follows thanks to \eqref{eq_conf_1}.
\end{proof}

The very same argument applied to the orientation covering also gives the corresponding result for M{\"o}bius bands.

\begin{cor}
Assume that $\Sigma_\eps$ is a M{\"o}bius band and let $\hat \Sigma_\eps$ be the orientation covering of $\Sigma_\eps$.
Let $\Phi \colon \overline{B}_1 \setminus B_r \to \hat \Sigma_\eps$ be a conformal homeomorphsim, which is smooth in the interior.
Then we have that $r \to 1$ as $\eps \to 0$.
\end{cor}

\begin{proof}[Proof of \cref{thm_annuli}]
We give the proof for $\Sigma$ an annulus, the case of M{\"o}bius bands is completely analogous.

We write $A_r = \overline{B_1} \setminus B_r \subset \IR^2$ and note that any compact annulus is conformal to some $A_r$ for a unique $r \in (0,1)$.
Consider the functional 
$\overline{\sigma}_1 \colon (0,1) \to (0,\infty)$ given by
$$
\overline{\sigma}_1(r) = \sup_\omega \sigma_1 \left(A_r,e^{2\omega} \xi \right) L_{e^{2 \omega} \xi}(\partial A_r),
$$
where $\xi$ denotes the canonical flat metric on $A_r$.
We want to show that $\sigma_1(r)>2\pi$ for any $r \in (0,1)$.

For $r$ sufficiently small, we have thanks to \cref{thm_comp_flat} that
\begin{equation} \label{eq_strict_1}
\overline{\sigma}_1(r) > 2\pi.
\end{equation}
Similarly, we have from \cref{thm_mp} and \cref{lem_conformal_2} that
\begin{equation}  \label{eq_strict_2}
\overline{\sigma}_1(r) > 2\pi
\end{equation}
for $r$ sufficiently close to $1$.
Moreover, we also know that
\begin{equation} \label{eq_limit}
\lim_{r \to 0} \overline{\sigma}_1(r) = \lim_{r \to 1} \overline{\sigma}_1(r) = 2\pi,
\end{equation}
see \cite[Proposition 4.4]{fs}.

Let $r_\star \in (0,1)$ be chosen such that $A_{r_\star}$ is conformal to the critical catenoid.
Suppose now towards a contradiction that there is some $s \in (0,1)$ such that $\overline{\sigma}_1(s) \leq 2 \pi$.
We assume that $s > r_\star $.
The argument for the other case is identical using \eqref{eq_strict_1} instead of \eqref{eq_strict_2}.
Then, thanks to \eqref{eq_strict_2} and \eqref{eq_limit}, we claim there has to be $t \in (s,1)$ such that
$$
\overline{\sigma}_1(t) = \max_{r \in [s,1)} \overline{\sigma}_1(r) > 2 \pi.
$$
Indeed, this follows immediately from the observation that $\overline{\sigma}_1$ when restricted to any of the open sets $\{ r \in (0,1) : \sigma_1(r) > 2\pi +\delta\}$ with $\delta>0$ is continuous thanks to \cref{thm_existence}.
Moreover, also thanks to \cref{thm_existence} we have $\omega \colon A_t \to \IR$ such that 
$$
\sigma_1(A_t,e^{2 \omega} \xi )L_{e^{2 \omega} \xi }(\partial A_t)
\geq 
\sigma_1(A_t,e^{2 \tau} \xi )L_{e^{2 \tau} \xi }(\partial A_t)
$$ 
for any smooth function $\tau \colon A_t \to \IR$. 
It then follows immediately that $e^{2 \omega} \xi$ is a local maximum of the normalized first Steklov eigenvalue and hence induced by a branched, free-boundary minimal immersion into $\mathbb{B}^N$ by first eigenfunctions thanks to \cref{thm_connection}.
But by the uniqueness of the critical catenoid, see \cref{thm_fs} above, this is impossible for $t \neq r_\star$.
\end{proof}

\bibliographystyle{alpha}
\bibliography{mybibfile}

\end{document}